\documentclass[11pt]{amsart}
\usepackage{amsfonts, amsmath, amssymb, amsthm, amscd, enumerate}
\usepackage[utf8]{inputenc}

\allowdisplaybreaks

\theoremstyle{plain}
\newtheorem{thm}{Theorem}[section]
\newtheorem{lem}[thm]{Lemma}
\newtheorem{prop}[thm]{Proposition}

\theoremstyle{remark}
\newtheorem{rem}[thm]{Remark}

\theoremstyle{definition}
\newtheorem{exam}[thm]{Example}

\newtheorem{dfn}[thm]{Definition}

\newcommand{\bbC}{\mathbb C}
\newcommand{\bbK}{\boldsymbol{\mathbb K}}

\newcommand{\bbR}{\mathbb R}
\newcommand{\bbS}{\mathbb S}
\newcommand{\bbZ}{\mathbb Z}

\newcommand{\bj}{\mathbf j}
\newcommand{\bs}{\mathbf s}
\newcommand{\bt}{\mathbf t}
\newcommand{\bx}{\mathbf x}

\newcommand{\cC}{\mathcal C}
\newcommand{\cG}{\mathcal G}
\newcommand{\cT}{\mathcal T}
\newcommand{\cV}{\mathcal V}

\newcommand{\cA}{\mathcal A}
\newcommand{\cJ}{\mathcal J}
\newcommand{\cK}{\mathcal K}
\newcommand{\cS}{\mathcal S}

\newcommand{\fU}{\mathfrak U}

\newcommand{\rA}{\mathrm A}
\newcommand{\rB}{\mathrm B}
\newcommand{\rD}{\mathrm D}
\newcommand{\rE}{\mathrm E}
\newcommand{\rG}{\mathrm G}

\newcommand{\rK}{\mathrm K}

\newcommand{\rM}{\mathrm M}
\newcommand{\rN}{\mathrm N}
\newcommand{\rT}{\mathrm T}
\newcommand{\rU}{\mathrm U}

\DeclareMathOperator{\Forall}{\forall}

\DeclareMathOperator{\ind}{ind}

\DeclareMathOperator{\id}{id}
\DeclareMathOperator{\tr}{tr}
\DeclareMathOperator{\ch}{ch}
\DeclareMathOperator{\td}{td}
\DeclareMathOperator{\AS}{AS}

\begin{document}
\title{K-theory and index formulas for boundary groupoid C${}^\ast$-algebras}

\author{Bing Kwan SO}
\thanks{Jilin University, 2699 Qianjin Ave., Changchun, China. 130012. e-mail: bkso@graduate.hku.hk}

\begin{abstract}
We compute explicitly the $\bbK$-groups of some boundary groupoid $C^*$-algebras with exponential isotropy subgroups.
Then we derive index formulas that computes the $\bbK$-theoretic and Fredholm indexes of elliptic (respectively totally elliptic)
pseudo-differential operators on these groupoids.
\end{abstract}

\maketitle

\section{Introduction}
The Atiyah-Singer index theorem is one of the most celebrated results of the twentieth century mathematics.
It states that given a compact manifold $\rM$:
\begin{enumerate}
\item
Any elliptic differential operator $D$ acting between sections of vector bundles over $\rM$ is invertible modulo compact operators,
hence $D$ is Fredholm;
\item
The Fredholm index of $D (D D^* + \id )^{-\frac {1}{2}}$ is the connecting map in $\bbK$-theory induced by the short exact sequence 
$$ 0 \to \cK \to \fU \to \fU / \cK \to 0 ;$$
\item 
One has $\bbK _0 (\cK) \cong \bbZ $, and moreover,
under such isomorphism, an explicit formula for the index, that produces an integer, is given:
$$\partial ([D (D D^* + \id )^{-\frac {1}{2}}]) = \int _\rM \AS (D) .$$
\end{enumerate}

There has been many attempts to generalize the Atiyah-Singer index theorem.
One major direction is suggested in \cite[Section 2]{Connes;Book}.
Here one considers any Lie groupoid $\cG \rightrightarrows \rM$ with $\rM$ compact 
(the classical case corresponds to the pair groupoid $\cG = \rM \times \rM$).
On $\cG$ one constructs the pseudo-differential calculus \cite{NWX;GroupoidPdO}.
The sub-algebras of order 0 and $-\infty$ operators can be completed to $C^*$-algebras 
$\fU (\cG)$ and respectively $C^* (\cG)$. 
Then one shows that any elliptic operator are invertible modulo order $-\infty $ order operators.
It follows for any order 0 elliptic operators one can consider its analytic index in $\bbK _0 (C^* (\cG))$ through the short exact sequence
$$ 0 \to C^* (\cG) \to \fU (\cG) \to \fU / C^* (\cG) \to 0.$$ 
Equivalently, the analytic index can be defined as in \cite{Nistor;Family}, or by the adiabatic groupoid \cite[Chapter 2]{Connes;Book}.
The objective of the index theorem is to give a alternative description of such index, preferably in ``topological'' terms.

Given a regular foliation $\mathcal F \subset T \rM$, 
one considers the holonomy groupoid $\cG$.
Generalizing the Atiyah-Singer index theorem,
one considers an embedding $\rM \to \bbR ^N$ for some sufficiently large $N$.
Then one can define a topological index using some Thom isomorphism and Morita equivalence (c.f. \cite[Chapter 2.9]{Connes;Book}).

The index theory for groupoids with leaves of different dimensions is less known.
The most well studied case is that of manifolds with corners.
The $\bbK$-theory of the interior of the $b$-stretched product $\cG \rightrightarrows \rM$ is computed in \cite{Monthubert;BdK}.
In the particular case of of manifolds with boundary, the $\bbK$ groups are trivial. 
An index theorem is again proven in \cite{Nistor;TopIndCorn}, using an auxiliary embedding of $\rM$ into a cube.
The key of the proof is that given any transverse embedding of manifold with corners $\rM_1 \to \rM _2$, 
one can construct explicitly, using the boundary defining functions,
a groupoid homomorphism from the $b$-stretched product of $\rM _1$ to $b$-stretched product of $\rM _2$.

On the other hand, one may also consider the Fredholm index of a Fredholm operator,
that is, elliptic operators invertible modulo $C^* (\rM_0 \times \rM_0) \cong \cK$,
the $C^*$-algebra of compact operators.
Here, recall \cite{Nistor;GeomOp} that an elliptic operator $\varPsi$ is Fredholm if $\varPsi $ is invertible over all singular leaves
(see \cite{Nistor;Fredholm1,Nistor;Fredholm2} for a discussion of the converse).
By using renormalized traces, or otherwise, 
one proves (see, for example \cite{Melrose;Book,Nistor;Hom2}) the Atiyah-Patodi-Singer index theorem in the case of manifolds with boundary, 
which is of the form 
$$\partial ([D (D D^* + \id )^{-\frac {1}{2}}]) = \int _\rM \AS (D) + \eta (D),$$
where $\eta$ is a non-local term that is, in general, not determined by the principal symbol of $D$.
One also gets similar results for some Lie manifolds \cite{Nistor;Hom2,Bohlen;HeatIndex} and the Bruhat sphere \cite{So;PhD}.

The notion of boundary groupoids is introduced \cite{So;FullCal}, as a generalization of manifolds with boundary,
Lie manifolds \cite{Nistor;Riem,Nistor;LieMfld}, and the symplectic groupoid of some Poison manifolds.
In the same paper, a calculus of pseudo-differential operators is constructed, which includes the finite rank parametrix of Fredholm operators,
and could be served as a framework for various problems (for example \cite{Nistor;Polyhedral3}).
Thus, it appears these groupoids are representative when one considers groupoids with leaves of different dimensions.

In this paper we study the index theory of elliptic operators on some boundary groupoids.
The main difficulty in establishing an abstract index theorem as in \cite{Nistor;TopIndCorn} is that, 
the structure of $\cG$ is not given by boundary defining functions, 
therefore it is not clear how an analogue of the embedding in \cite{Nistor;TopIndCorn} can be constructed.
Instead we attempt to first compute explicitly the $\bbK$-groups of $C^* (\cG)$,
and then give a formula for the index of an elliptic operator.

To calculate the $\bbK$-groups we use the composition series \cite{Nistor;Family}.
By definition a boundary groupoid is the disjoint union of finite number of products of pair groupoids and Lie groups:
$\cG = \bigcup _i \rG _i \times \rM _i \times \rM _i $.
Therefore one naturally constructs a composition series, from which one obtains the corresponding exact sequences.
Assuming further that the isotropy subgroups $\rG _i$ are exponential (i.e., connected, simply-connected and solvable Lie groups), 
then these exact sequences are particularly simple (see Equations \eqref{Comp3}, \eqref{Comp4}).

As a first application of the composition series, 
we shall observe in Section 4 that the composition series
\begin{equation}
0 \to C^* (\rM _0 \times \rM _0) \cong \cK \to C^* (\cG) \to C^* (\cG) / C^* (\rM _0 \times \rM _0) \to 0 
\end{equation}
induces a natural map from the Fredholm index to the $\bbK _0 (C^*(\cG))$ index,
hence unifying the two index problems. 

Then we shall turn to the simplest examples
$$\cG = \rM _0 \times \rM _0 \cup \bbR ^q  \times \rM _i \times \rM _1 ,$$
for each $q$. 
We shall consider the cases $q > 1$ odd, $q = 1$, and $q$ even respectively in Sections 5, 6, 7. 

We consider the $q > 1$ odd case first, in Section 5, because the computation involves some techniques that is useful later. 
To compute the $\bbK$-groups we first use the same arguments as in \cite[Proposition 2]{Nistor;TopIndCorn} to show that, 
if $\rM _0$ and $\rM _1$ are both contractible then
$$ \bbK _0 (C^* (\cG)) \cong \bbZ , \quad \bbK _1 (C^* (\cG)) \cong \bbZ.$$
We then show that the same result holds in general by Morita equivalence.
Moreover, one has an isomorphism $\bbK _0 (\cK) \to \bbK _0 (C^* (\cG))$,
in other words, the Fredholm index is in this case equivalent to the $\bbK$-theoretic index.
With this result in mind, we turn to consider the index of elliptic operators.
We shall consider two special cases,
when $\rG _1$ is Euclidean and $\rM _1$ is a point some topological arguments similar to \cite{Nistor;Perturbed} can be used;
Alternatively, when the singular structure is similar to the very small calculus,
a renormalized index along the lines of \cite{Bohlen;HeatIndex,Nistor;Hom2} can be defined, 
and we shall show in addition that the $\eta$ term vanishes, leaving only the Atiyah-Singer term.

Much about the case when $q=1$, and when $\cG$ is induced by the very small calculus, is known.
Using the same arguments as in Section 5, we show that when $\rM$ is a manifold with boundary
$$ \bbK _0 (C^* (\cG)) \cong \{ 0 \} , \quad \bbK _1 (C^* (\cG)) \cong \{ 0 \}.$$
However when $\rM$ is a connected manifold without boundary partitioned by a co-dimensional 1 sub-manifold $\rM _1$
(i.e., $\rM _0 \setminus \rM _1$ has two components),
then $\bbK _0 (C^* (\cG)) \cong \bbZ$ and it is meaningful to consider the index in $\bbK _0 (C^* (\cG))$.
We shall see in Theorem \ref{OneInd} that the $\bbK$-theoretic index is the {\it difference} of the Fredholm indexes.

We consider the even dimensional isotropy subgroup case in Section 7.
We shall see that in this case the $\bbK$-theoretic index contains the Fredholm index.

In section 8, we compute the $\bbK$-groups of some examples with $r > 1$.
We do so by filling results from Sections 5-7 into the composition series.

Preceding these computations, we briefly review some background in groupoid pseudo-differential operators and $C^*$-algebras in Section 2,
and we shall recall the definition of boundary groupoids in Section 3. We point out some open problems in the last Section 9.

\section{Composition series for the groupoid $C^*$-algebra}
In this section, we review the results of \cite{Nistor;Family}.
Let $\cG \rightrightarrows \rM $ be a Hausdorff Lie groupoid.
We denote by $\bs$ and $\bt$ respectively the source and target maps.

\subsection{Groupoid $C^*$-algebras}
Let $\rE \to \rM$ be a vector bundle. 
Recall \cite{NWX;GroupoidPdO} that an order $m$ pseudo-differential operator on $\cG$ is a right invariant, 
smooth family $\varPsi = \{ \varPsi _x \} _{x \in \rM}$,
where each $\varPsi _x $ is an order $m$ classical pseudo-differential operator on sections of $\bt ^* \rE \to \bs ^{-1} (x)$.
We denote by $\Psi ^m (\cG, \rE)$ (resp. $\rD ^m (\cG , \rE)$) the algebra of uniformly supported, 
order $m$ classical pseudo-differential operators (resp. partial differential operators).

Suppose that $\rM$ is compact.
Then there is an unique (up to equivalence) Riemannian metric on its Lie algebroid $\rA \to \rM$,
which induces a family of right invariant Riemannian metrics on $\bs ^{-1} (x), x \in \rM$.
We denote by $\mu _x$ the Riemannian volume forms (which is a Haar system on $\cG$), and $d (\cdot, \cdot)$, the Riemannian distance function.

Recall \cite{Nistor;Family} that one defines the strong norm
$$ \| \varPsi \| := \sup _\rho \| \rho (\varPsi ) \|, $$
where $\rho $ ranges from all bounded $*$-representations of $\Psi ^0 (\cG, \rE)$ satisfying
$$ \| \rho (\varPsi) \| \leq \sup _{x \in \rM} 
\Big\{ \int _{\bs ^{-1} (x)} | \kappa _\varPsi (g) | \mu _x , \int _{\bs ^{-1} (x)} | \kappa _\varPsi (g ^{-1}) | \mu _x \Big\},$$
whenever $\varPsi \in \Psi ^{- \dim \rM - 1} (\cG, \rE)$ with (continuous) kernel $\kappa _\varPsi$.

\begin{dfn}
The $C^*$-algebras $\fU (\cG)$ and $C ^* (\cG)$ are defined to be the completion of $\Psi ^0 (\cG, \rE)$ (respectively $\Psi ^{- \infty} (\cG, \rE)$)
with respective to the strong norm $\| \cdot \| $.
\end{dfn}

One also defines the reduced $C^*$-algebras $\fU _r (\cG), C^* _r (\cG)$ by completing 
$\Psi ^0 (\cG, \rE)$ (respectively $\Psi ^{- \infty} (\cG, \rE)$) with respect to the reduced norm
$$ \| \varPsi \| _r := \sup _{x \in \rM} \big\{ \| \varPsi _x \| _{L^2 (\bs ^{-1} (x))} \big\}.$$

Suppose further that $\cG$ has polynomial volume growth.
Then convolution product is well defined between kernels with faster than polynomial decay, and hence one may define:
\begin{dfn}
The sub-algebra of Schwartz kernels $S ^* (\cG) \subset C ^* _r (\cG)$ is the set of all kernels $\psi \in C^*(\cG)$ satisfying:
\begin{enumerate}
\item
$\psi $ is continuous;
\item
$\psi |_{\cG _i}$ is smooth for all $i$;
\item
For any collection of sections $V_1, \cdots, V_k \in \Gamma ^\infty (\rA)$, 
regarded as first order differential operators (right invariant vector fields) on $\cG$ and $N =0, 1, 2, \cdots$,
$$ a \mapsto (V_1 \cdots V_i \psi V _{i+1} \cdots V _k )(d (a, \bs (a)) + 1)^{-N}$$
is a bounded function.
\end{enumerate}
\end{dfn}

The significance of $S ^* (\cG)$ lies in (c.f. \cite[Section 6]{Nistor;Funct}, \cite{Bohlen;HeatIndex}):
\begin{lem}
\label{FunctCal}
The sub-algebras $S ^* (\cG) \subset C^* _r (\cG)$ and $\Psi ^0 (\cG , \rE) \oplus S ^* (\cG) \subset \fU^* _r (\cG)$ are dense, 
and are invariant under holomorphic functional calculus.
\end{lem}

\subsection{Invariant sub-manifolds and the composition series}
Let $\cG$ be a groupoid.
\begin{dfn}
We say that an embedded sub-manifold $\rM' \subset \rM$ is an {\it invariant sub-manifold} if 
$$ \bs^{-1} (\rM') = \bt ^{-1} (\rM' ) .$$
For any invariant sub-manifold $\rM '$, 
we denote the sub-groupoid $ \bs^{-1} (\rM ') = \bt ^{-1} (\rM ') $ by $\cG _{\rM '}$. 
\end{dfn} 

\begin{dfn}
Given a closed invariant sub-manifold $\rM' $ of $\cG$.
For any $\varPsi = \{ \varPsi _x \} _{x \in \rM} \in \Psi ^m (\cG, \rE)$, define the restriction
$$ \varPsi |_{\cG _{\rM'}} := \{ \varPsi _x \} _{x \in \cG _{\rM'}} \in \Psi ^m (\cG _{\rM'}, \rE).$$
Restriction extends to a map from $\fU (\cG) $ to $\fU (\cG _{\rM'})$ and also from $C^* (\cG)$ to $C^* (\cG _{\rM'})$.
\end{dfn}

Now suppose we are given a groupoid $\cG \rightrightarrows \rM$ ($\rM $ not necessarily compact),
and closed invariant sub-manifolds 
$$ \rM = \bar \rM _0 \supset \bar \rM _1 \supset \cdots \supset \bar \rM _r .$$
For simplicity we denote $\bar \cG _i := \cG _{\bar \rM _i}$.

\begin{dfn}
Let $\sigma : \Psi ^m (\cG) \to C^\infty (S \rA')$ denotes the principal symbol map
(where $S \rA'$ denotes the sphere sub-bundle of the dual of the Lie algebroid of $\cG$).
For each $i = 1, \cdots , r$, define the joint symbol maps
\begin{equation}
\label{JointSym}
\bj _i : \Psi ^m (\cG) \to C^\infty (S \rA') \oplus \Psi ^m (\bar \cG _i),
\quad \bj _i (\varPsi ) := (\sigma (\varPsi ) , \varPsi |_{\bar \cG _i}).
\end{equation}
The map $\bj _i$ extends to a homomorphism from $\fU (\cG)$ to $C ^0 (S \rA ^*) \oplus \fU (\bar \cG _i)$.

We say that $\varPsi \in \Psi ^m (\cG)$ is elliptic if $\sigma (\varPsi)$ is invertible,
and it is fully elliptic if $\bj _1 (\varPsi )$ is invertible.
\end{dfn}

\begin{dfn}
Denote by $\cJ _0 := \overline {\Psi ^{-1} (\cG)} \subset \fU (\cG)$, 
and $\cJ _i \subset \fU (\cG), i = 1, \cdots , r$ the null space of $\bj _{r - i + 1}$.
\end{dfn}

By construction 
$$ \cJ _0 \supset \cJ _1 \supset \cdots \supset \cJ _r .$$
Also, any uniformly supported kernels in $\Psi ^{- \infty} (\cG _{\bar \rM _i \setminus \bar \rM _j}, \rE |_{\bar \rM _i \setminus \bar \rM _j})$,
can be extended to a kernel in $\Psi ^{- \infty }  (\cG _{\bar \rM _i}, \rE |_{\bar \rM _i })$ by zero.
This induces a $*$-algebra homomorphism from $C^* (\cG _{\bar \rM _i \setminus \bar \rM _j} ) $ to $ C^* (\cG _{\bar \rM _i}) $.
The key fact we shall use is:
\begin{lem}
\cite[Lemma 2 and Theorem 3]{Nistor;Family}
\label{Comp}
One has short exact sequences
\begin{align}
& 0 \to \cJ _{i+1} \to \cJ _i \to C ^* (\cG _{\bar \rM _i \setminus \bar \rM _{i +1}}) \to 0 \\
& 0 \to C^* (\cG _{\bar \rM _i \setminus \bar \rM _j} ) \to C^* (\bar \cG _i) \to C^* (\bar \cG _j) \to 0, \quad \Forall j > i.
\end{align}
\end{lem}

\section{The boundary groupoid and its composition series}

\begin{dfn}
\label{Dfn}
Let $\cG \rightrightarrows \rM$ be a Lie groupoid with $\rM$ compact.
We say that $\cG$ is a boundary groupoid if:
\begin{enumerate}
\item
the singular foliation defined by the anchor map $\nu : \rA \to \rT \rM $ has finite number of leaves 
$\rM _0 , \rM _1, \cdots , \rM _r \subset \rM$ 
(which are invariant sub-manifolds),
such that $\dim \rM = \dim \rM _0 > \dim \rM _1 > \cdots > \dim \rM _r $;
\item
For all $k = 0, 1, \cdots , r$, $\bar \rM _k := \rM _k \cup \cdots \cup \rM _r $ are closed sub-manifolds of $\rM$;
\item
$\cG _0 := \cG _{\rM _0}$ is the pair groupoid, 
and $\cG _k := \cG _{\rM _k} \cong \rG _k \times \rM _k \times \rM _k $ for some Lie group $\rG _k$;
\item
For each $k$, there exists an unique sub-bundle $\bar \rA _k \subset \rA |_{\bar \rM _k}$ such that 
$\bar \rA _k |_{\rM _k} = \ker (\nu |_{\rM _k}) $ ($= \mathfrak g _k \times \rM _k$).
\end{enumerate}
\end{dfn}

Before proceeding, we give an example of a boundary groupoid:
\begin{exam}
\label{BruhatExam}
\cite{Lu;PoissonCohNotes}
Let $\rG = \rK \rA \rN $ be the Iwasawa decomposition,
where 
\begin{align*}
\rG = & \mathbb {S L} (n + 1, \bbC) \\ 
\rK = & \mathbb {S U} (n + 1) \\ 
\rA = & \{ \text{ real diagonal matrices with unit determinant } \} \\
\rN = & \{ \text{ upper triangular matrices with unit diagonal } \}.
\end{align*}
Let $\rT := S (\mathbb U (1) \times \mathbb {S U} (n)) = \left\{ \left(
\begin{array}{ll}
a _{1 1} & 0 \\
0 & * 
\end{array}
\right) \in \mathbb {S U} (n + 1) \right\}$, 
$\rN _0 :=
\left\{ \left(
\begin{array}{ll}
1 & * \\
0 & I 
\end{array}
\right) \right\}
$
and define the left action of $\rT$ on $\rK \times \rN _0$ by
$$ g \cdot (k , n) := (g k, g n g ^{-1}), \quad \forall (k, n) \in \rK \times \rN, g \in \rT .$$
Define the groupoid operations on $\cG := \rT \backslash (\rK \times \rN _0) \rightrightarrows \rT \backslash \rK $:
\begin{align*}
\text {source and target maps: } & \bs ( {}_\rT (k, n) ) = {}_ \rT k , \bt ( {}_\rT (k, n) ) := {}_\rT k' , \\
& \text {where } n k = k' a' n' \text { is the Iwasawa decomposition;} \\  
\text {multiplication: } &  \mathbf m ({}_\rT (k _1 , n _1) , {}_ \rT (k _2 , n _2) ) := {}_\rT (k _2 , n _1 n _2) , \\
& \text {provided one has Iwasawa decomposition } n _2 k _2 = k _1 a' n' ; \\
\text {inverse: } & \mathbf i ( {}_\rT (k, n)) := {}_\rT (k', n^{-1}) , \\
& \text {where } n k = k' a' n' \text { is the Iwasawa decomposition;} \\  
\text {unit: } & \mathbf u ({}_ \rT k ) := {}_ \rT (k , e) , e \in \rN _0.
\end{align*} 
Note that $\rT \backslash \rK $ naturally identifies with $\mathbb {C P} (n)$.
Indeed $\cG$ is the symplectic groupoid of the Bruhat Poisson structure on $\mathbb {C P} (n)$.

To see that $\cG$ is a boundary groupoid, one considers the isotropy subgroups:
Given $k \in \rK$, suppose that there exists $n \in \rN _0$ such that one has Iwasawa decomposition
$$ n k = k a' n' .$$
This implies 
$$ k ^{-1} n k = a' n'.$$
Writing 
$ k = (a _{i j} ) = (k _1, \cdots, k _{n+1}), n =
\left(
\begin{array}{ll}
1 & y \\
0 & I 
\end{array}
\right)
$,
then the left hand side equals
$$ I + (b _{i j})$$
where $b _{i j} = a _{1 i} (0, y) \cdot k _j$.
Since the right hand side is upper triangular with unit determinant, if, for example,
$a _{1 n+1} \neq 0$, then $y = 0$. 
It follows that the groupoid over the invariant sub-manifold $\{[*, 1] \} \subset \mathbb {C P} (n)$ is the pair groupoid.
\end{exam}

\begin{dfn}
\label{Central}
We say that a boundary groupoid $\cG$ is {\it strongly central} if there exists boundary groupoids 
$$ \cG ^0 _1 \rightrightarrows \bar \rM _1 , \cdots , \cG ^0 _r \rightrightarrows \bar \rM _r ,$$
and Lie groups $\rG ^0 _1, \cdots \rG ^0 _r$ such that for all $i$,
\begin{align}
\bar \cG _i =& \; (\rG ^0 _1 \times \cdots \times \rG ^0 _i) \times \cG ^0 _i  \\ \nonumber
\cG ^0 _i |_{\bar \rM _{i+1}} =& \; \rG ^0 _{i+1} \times \cG ^0 _{i+1}.
\end{align}
\end{dfn}

Note that the definition above implies the isotropy subgroup over $\rM _i$ is $\rG _i = \rG ^0 _1 \times \cdots \times \rG ^0 _i.$

Given a boundary groupoid, we consider the sequence of invariant sub-manifolds, 
$ \rM \supset \bar \rM _1 \supset \cdots \supset \bar \rM _r, $ 
where $\bar \rM _i$ is given to be (2) of Definition \ref{Dfn}.
The second short exact sequence of Lemma \ref{Comp} then induces the six-terms exact sequences:
\begin{equation}
\label{Comp2}
\begin{CD}
\bbK _1 (C^* (\cG _{\rM _i} )) @>>> \bbK _1 (C^* (\bar \cG _i)) @>>> \bbK _1(C^* (\bar \cG _{i+1})) \\
@AAA @. @VVV \\
\bbK _0(C^* (\bar \cG _{i+1})) @<<< \bbK _0 (C^* (\bar \cG _i)) @<<< \bbK _0 (C^* (\cG _{\rM _i }))
\end{CD}
\end{equation}
The system \eqref{Comp2} greatly simplifies if $r = 1$, and if $\rG _1$ is solvable, connected and simply connected (i.e. exponential).
Here we only have one six-terms exact sequence, 
and we have 
\begin{align*}
C^* (\cG _{\rM _0}) = C^* (\rM _0 \times \rM _0 ) \cong & C^* _r (\rM _0 \times \rM _0 ) \cong \cK \\
C^* (\bar \cG _0) = & C^* (\cG)  \\
C^* (\bar \cG _1) = C^* (\cG _1) = & C^* (\rG _1 \times \rM _1 \times \rM _1) \cong \cK \otimes C^* (\rG _1)
\end{align*}
(the last isomorphism follows from \cite[Proposition 2]{Nistor;GeomOp}).
The $\bbK$ groups of $\cK$ are well known --- $\bbK _0 (\cK) \cong \bbZ, \bbK _0 (\cK) \cong \{0 \}$,
and one uses Connes' Thom isomorphism to compute the $\bbK$ groups of $\cK \otimes C^* (\rG _1)$.
Then \eqref{Comp2} reads
\begin{align}
\label{Comp3}
\begin{CD}
\bbZ @>>> \bbK _1 (C^* (\cG)) @>>> \bbZ \\
@AAA @. @VVV \\
0 @<<< \bbK _0 (C^* (\cG)) @<<< 0
\end{CD}
& \quad \text{ if $\dim \rG_1$ is even,} \\
\label{Comp4}
\begin{CD}
\bbZ @>>> \bbK _1 (C^* (\cG)) @>>> 0 \\
@AAA @. @VVV \\
\bbZ @<<< \bbK _0 (C^* (\cG)) @<<< 0
\end{CD}
& \quad \text{ if $\dim \rG_1$ is odd.}
\end{align}
In sections 5-7 we shall explicitly compute the $\bbK$ groups of $C^* (\cG)$, corresponding to different dimension of $\rG _1$.

\section{The relation between Fredholm and $\bbK _0 (C ^* (\cG))$ index}
We turn to study the relation between Fredholm and $\bbK _0 (C ^* (\cG))$ index.
Recall \cite{Nistor;GeomOp} that an elliptic operator $\varPsi$ is Fredholm if $\varPsi |_{\bar \cG _1}$ is invertible,
and its Fredholm index is the connecting map induced by the short exact sequence
$$ 0 \to C^* (\rM _0 \times \rM _0 ) \cong \cK \to \fU \to \fU / \cK \to 0 .$$
Consider the morphism of short exact sequence
\begin{equation}
\label{UMorph}
\begin{CD}
0 @>>> C^* (\cG _0 ) @>>> \fU @>>> \fU /  C^* (\cG _0 ) @>>> 0 \\
@. @VVV @| @VVV @. \\
0 @>>> C^* (\cG) @>>> \fU @>>> \fU / C^* (\cG) @>>> 0
\end{CD}
\end{equation}
where the left column is the incursion map of Lemma \ref{Comp}.
The right column is clearly defined by quotient out a bigger ideal, hence fits into the short exact sequence
$$ 0 \to \fU /  C^* (\cG _0 ) \to \fU / C^* (\cG) \to C^* (\cG) / C^* (\cG _0) \cong C^* (\bar \cG _1) \to 0.$$
By naturality of the index one gets:
\begin{equation}
\begin{CD}
\label{Surj}
\bbK _1 (C^* (\bar \cG _1)) @>>> \bbK _1 (\fU / C^ *(\cG _0 )) @>>> \bbK _1 (\fU / C^* (\cG)) @>>> \bbK _0 (C^* (\bar \cG _1))  \\
@| @VVV @VVV @| \\
\bbK _1 (C^* (\bar \cG _1)) @>>> \bbK _0 ( C^ *(\cG _0 )) @>>> \bbK _0 (C^* (\cG)) @>>> \bbK _0 (C^* (\bar \cG _1))
\end{CD}.
\end{equation}
In other words, a class $[\varPsi] \in \bbK _1 (\fU / C^* (\cG))$ is represented by a Fredholm operator $[\tilde \varPsi]$
if and only if its index in $\bbK _0 (C^* (\bar \cG _1))$ is zero,
and in this case its index in $ \bbK _0 (C^* (\cG))$ is the image of the Fredholm index of $[\tilde \varPsi]$.

\section{Isotropy subgroup of odd dimension $> 1$}
We begin with the case $\cG = (\rM _0 \times \rM _0 ) \cup (\bbR ^q \times \rM _1 \times \rM _1)$,
where $q > 1$ is odd.

\subsection{The classifying spaces}
In \cite{Nistor;TopIndCorn}, the authors use the cube with appropriate faces removed as classifying spaces.
Here, we consider examples of a similar nature. Thus, we suppose
$$\cG = (\rM _0 \times \rM _0 ) \cup (\rG _1 \times \rM _1 \times \rM _1),$$
such that both $\rM _0 $ and $\rM _1$ are contractible, and $\rG _1$ is exponential with dimension $q > 1$.
Since $\cG$ acts on itself and has contractible $\bs$-fibers by hypothesis, 
$\rE \cG = \cG$ and $\rB \cG = \rE \cG / \cG$ is a classifying space of $\cG$ in the sense of \cite[Chapter 2 8.$\gamma$]{Connes;Book}.

\begin{exam}
\label{OddSphere}
Let $\rM  = \bbS ^q, p \in \rM $ and $\rM _1 = \{ p \}$.
Observe that $\rM _0 = \bbS ^q \setminus \{ p \} $ is Euclidean.
More generally, one may consider $\cG = (\rM _0 \times \rM _0 ) \cup (\rG _1 \times \rM _1 \times \rM _1)$,
where $\rM = \rU \times \bbS ^q, p \in \bbS ^q $, $\rU$ is contractible, and $\rM _1 = \rU \times \{ p \}$.
\end{exam}

We compute the $\bbK$-groups of $C^* (\cG)$ using arguments similar to \cite{Nistor;TopIndCorn}.
First, we briefly recall the tangent groupoid construction of the analytic index \cite{Connes;Book}.
Given any Lie groupoid $\cG$, denote by $\rA$ its Lie algebroid with anchor map $\nu$.
The tangent groupoid is by definition
$$ \cT := (\rA \times \{ 0 \}) \cup (\cG \times (0,1]) \rightrightarrows \rM \times [0, 1] .$$
It has closed invariant sub-manifolds 
$$ \rA \times \{ 0 \} \rightrightarrows \rM \times \{ 0 \}, \quad 
\cG \times \{ 1 \} \rightrightarrows \rM \times \{ 1 \}.$$
Here $\rA \times \{ 0 \} \rightrightarrows \rM \times \{ 0 \} = \rA \rightrightarrows \rM $ is just the vector space $\rA$,
regarded as a bundle of groups.
Recall that one naturally identifies $C^* (\rA)$ with $C ^0 _0 (\rA ^*)$ using Fourier transform.

By \cite[Theorem 3]{Nistor;Family} again, restriction to invariant sub-manifolds give short exact sequences
\begin{align}
\label{evmaps}
0 \to C^* (\cT |_{\rM \times (0, 1]}) \to C^* (\cT) & \xrightarrow{e ^0} C^* (\rA) \to 0 \\ \nonumber
0 \to C^* (\cT |_{\rM \times [0, 1)}) \to C^* (\cT) & \xrightarrow{e ^1} C^* (\cG) \to 0.
\end{align}
Since $C^* (\cT |_{\rM \times (0, 1]}) = C^* (\cG) \otimes C^0 _0 ((0,1])$ is contractible, 
$e ^0 _* : \bbK _\bullet ( C^* (\cT)) \to \bbK _\bullet ( C^* (\rA))$ is an isomorphism.
The analytic index map is defined as the composition
$$ \ind := e ^1 _* \circ (e ^0 _* ) ^{-1} : \bbK _\bullet (C ^* (\cA)) \to \bbK _\bullet  (C^* (\cG)).$$

By the same method of proof of \cite[Proposition 2]{Nistor;TopIndCorn}, we have
\begin{lem}
\label{OddClassSp}
For any $\cG = (\rM _0 \times \rM _0 ) \cup (\rG _1 \times \rM _1 \times \rM _1)$,
such that both $\rM _0 $ and $\rM _1$ are contractible,
$ \ind : \bbK _\bullet (C ^* (\rA)) \to \bbK _\bullet (C^* (\cG))$ is an isomorphism.
\end{lem}
\begin{proof}
Consider the analytic index maps for $\cG$ and the invariant sub-manifolds $\rM _0, \rM _1$.
Clearly the tangent groupoid construction and evaluation maps in \eqref{evmaps} are compatible with restriction to invariant sub-groupoids
$\cG _0, \cG _1$.
It follows that one has a morphism of six terms exact sequence:
$$
\begin{CD}
@>>> \bbK _\bullet (C^* (\cG _0)) @>>> \bbK _\bullet (C^* (\cG)) @>>> \bbK _\bullet (C^* (\cG _1)) @>>> \\
@. @A{\ind}AA @A{\ind}AA @A{\ind}AA @. \\
@>>> \bbK _\bullet (C^* (\rA |_{\rM _0})) @>>> \bbK _\bullet (C^* (\rA)) @>>> \bbK _\bullet (C^* (\rA |_ {\rM_1})) @>>> 
\end{CD}
$$
Since by assumption $\rM _0 $ and $\rM _1$ are contractible, the analytic index gives isomorphisms
$$ \bbK _\bullet (C^* (\rA |_{\rM _0})) \cong \bbK _\bullet (C^* (\cG _0)), \quad
\bbK _\bullet (C^* (\rA |_{\rM _1})) \cong \bbK _\bullet (C^* (\cG _1)).$$
It follows from the five lemma that 
$$\ind : \bbK _\bullet (C^* (\rA)) \to \bbK _\bullet (C^* (\cG )) $$
is also an isomorphism.
\end{proof}

\begin{exam}
\label{OddEx}
For the groupoid in Example \ref{OddSphere} with $\dim \rG _1 > 1$ odd,
we get 
$$ \bbK _0 (C^* (\rA )) \cong \bbZ \cong \bbK _0 (C^* (\cG)), \quad \bbK _1 (C^* (\rA )) \cong \bbZ \cong \bbK _1 (C^* (\cG))$$
using Lemma \ref{OddClassSp}.
\end{exam}

\subsection{The general case and Morita equivalence}
In example \ref{OddSphere}, we considered groupoids of the form
\begin{equation}
\label{ClassGpoid}
\tilde \cG = (\tilde \rM _0 \times \tilde \rM _0 ) \cup (\rG _1 \times \tilde \rM _1 \times \tilde \rM _1) \rightrightarrows \tilde \rM,
\end{equation}
where $\tilde \rM = \rU \times \bbS ^q, p \in \bbS ^q $, $\rU$ is contractible, and $\tilde \rM _1 = \rU \times \{ p \}.$
Given such $\cG$, we fix an arbitrary connected open neighborhood $\rU' \subset \bbS ^q$ of $p $,
and consider
$$\cG' := \bs ^{-1} (\rM') \cap \bt ^{-1} (\rM'),$$
where $\rM' := \rU \times \rU' \subset \rM$.

\begin{lem}
\label{Morita}
The Lie groupoid $\cG'$ is Morita equivalent to $\tilde \cG$.
\end{lem}
\begin{proof}
Using the same arguments as the proof of \cite[Lemma 6]{Nistor;TopIndCorn} and \cite[Proposition 3]{Nistor;TopIndCorn},
one sees that $\rM' \xleftarrow{\bt} \bt ^{-1} (\rM') \subset \cG \xrightarrow{\bs} \rM$ defines the required Morita morphism.
\end{proof}

We turn to the general case.
\begin{thm}
\label{OddMain}
For any groupoid of the form 
$$\cG = (\rM _0 \times \rM _0 ) \cup (\rG _1 \times \rM _1 \times \rM _1),$$
with $\rG _1$ exponential and $\dim \rG _1 = q > 1 $ odd,
one has 
$$ \bbK _0 (C^* (\cG)) \cong \bbZ , \quad \bbK _1 (C^* (\cG)) \cong \bbZ.$$
\end{thm}
\begin{proof}
We prove the theorem by constructing a groupoid $\tilde \cG$ of the form \eqref{ClassGpoid}, 
and then prove that $\cG' $ is Morita equivalent to $\cG$.

Fix any $p \in \rM _1$, and a contractible open neighborhood $\rU$ of $p$ in $\rM _1$.
Take a tubular neighborhood $\rM'$ of $\rU$ in $\rM$,
i.e., $\rM' $ is diffeomorphic to $\rU \times \rU' $ for some open ball $\rU'$ (which is contractible).
Define $\cG' :=  \bs^{-1} (\rM') \cap \bt ^{-1} (\rM')$. 

Fix any embedding $\rU' \to \bbS ^q $ as a open subset.
Taking Cartesian product with $\rU$, one gets an embedding $\rM' \to \tilde \rM := \rU \times \bbS ^q$.
Denote by $\cV'$ the structural vector field of $\rM'$ (the image of the anchor map). 
Consider 
$$ \tilde \cV := \{ X \in \Gamma ^\infty (T \tilde \rM) : X |_{\rM'} \in \cV' \}. $$
Clearly $\tilde \cV$ is locally free of rank $\dim \tilde \rM$ (c.f. \cite[Condition (1)]{Nistor;Riem}), i.e.,
for each $x \in \tilde \rM$ there exist $\theta \in C^\infty (\tilde \rM)$, $\theta (x) =1$,
and vector fields $X_1 \cdots X_{\dim \tilde \rM} \in \tilde \cV$ such that
for any $X \in \tilde \cV$, there exists smooth function $f_1 , \cdots, f _{\dim \tilde \rM}$,
whose germs at $x$ are uniquely determined, such that
$$\theta (f _1 X_1 + \cdots + f _{\dim \tilde \rM} X _{\dim \tilde \rM} - X) = 0 .$$
It follows \cite{Nistor;Riem} that there exists a Lie algebroid $\tilde \rA \to \tilde \rM $ with anchor map $\tilde \nu$ such that
$\tilde \cV = \tilde \nu (\Gamma ^\infty (\tilde \rA))$. 
Moreover, $\tilde \rA |_{\rM'}$ is isomorphic to $\rA'$.

One integrates $\tilde \rA$ to a (unique) groupoid $\tilde \cG$ with connected and simply connect $\bs$-fibers
(c.f. \cite{Nistor;IntAlg'oid,Debord;IntAlgebroid,Fern'd;IntAlgebroid}).
Since the orbits of $\tilde \rA$, diffeomorphic to $\{ p \} \times \rU $ and $(\bbS ^q \setminus \{ p \}) \times \rU$,
are both simply connected,
$\tilde \cG$ is necessarily of the form
$$ \tilde \cG = (\tilde \rM _0 \times \tilde \rM _0 ) \cup (\rG _1 \times \rU \times \rU). $$
Also, it is clear that
$\cG'$ is isomorphic to $ \bs^{-1} (\rM') \cap \bt ^{-1} (\rM') \subset \tilde \cG$.
It follows form Lemma \ref{Morita} that $\cG$, $\cG'$ and $\tilde \cG$ are all Morita equivalent.
Hence they have isomorphic $\bbK$ groups, which we have computed in Example \ref{OddEx}.
\end{proof}

\begin{rem}
\label{MainRem}
Putting the results of Theorem \ref{OddMain} into \eqref{Comp4},
we get the exact sequence
\begin{equation}
\label{OddIsom}
0 \to \bbK _1 (C^* (\cG)) \cong \bbZ \to \bbK _1 (C ^* (\rG _1) \otimes \cK) \to \bbK _0 (\cK) \cong \bbZ \to \bbK _0 (C^* (\cG)) \cong \bbZ \to 0 ,
\end{equation}
which implies the inclusion form $\bbK _0 (\cK) $ to $\bbK _0 (C ^* (\cG))$ is an isomorphism.
The commutative diagram \eqref{Surj} now reads
\begin{equation}
\begin{CD}
\bbK _1 (\fU / C^*(\cG _{\rM _0})) @>>> \bbK _1 (\fU / C^* (\cG)) \\
@VV{\partial {}}V @VV{\partial {}}V  \\
\bbK _0 (C^*(\cG _{\rM _0})) @>{\cong}>> \bbK _0 (C^* (\cG)) 
\end{CD}.
\end{equation}
It follows that unlike the manifold with boundary case, for elliptic Fredholm operators, the $\bbK$ theoretic index {\it is} the Fredholm index.
Moreover, the Fredholm index of such operators only depends on its principal symbol, as opposed to full symbol.
\end{rem}

\subsection{An index formula for the asymptotic Abelian case $\cG = \rM _0 \times \rM _0 \cup \bbR ^q$}
Knowing from Example \ref{OddEx} that $\bbK _0 (C^* (\cG )) \cong \bbZ$,
we turn to compute an index formula.
Here, we shall adopt the method of \cite{Nistor;Perturbed}.

Consider the image of the joint symbol map $\bj _1$ in \eqref{JointSym}.
Using Fourier transform 
$\Psi ^0 (\bar \cG _1) = \Psi ^0 (\bbR ^q)$ is isomorphic to a sub-algebra of $C ^0 (\bar \rA' |_{p_0})$,
the space of continuous functions (matrix valued sections) on the radial compactification of $\rA' |_{p_0}$.
Through this identification, the image of $\bj _1$ lies in the space
$$ C ^0 (S \rA' \cup \bar \rA' |_{p_0})
= \{ (f, f _1) \in C ^0 (S \rA') \oplus C ^0 (\bar \rA' |_{p_0}) : f |_{S \rA' _{p_0}} = f_1 |_{S \rA' _{p_0}} \}.$$
By considering homotopy of the underlying topological spaces,
it follows that every elliptic operator $\varPsi \in \Psi ^0 (\cG)$ with invertible total symbol is homotopic to some
$\varPsi' \in \Psi ^0 (\cG)$,
such that $\varPsi'|_{\cG _1}$ is multiplication by a constant tensor (i.e., the Fourier transform of $\varPsi'|_{\cG _1}$ is constant).

The operator $\varPsi'$ is in turn homotopic to an operator of the form
$$ A + \varPsi'' ,$$
where $A$ is a tensor and $\varPsi'' \in \fU (\rM_0 \times \rM _0) \subset \fU (\cG)$ having compact support.

Through the connecting map of the short exact sequence
$$ 0 \to C ^0 (\rA' |_{\rM _0}) \to C ^0 (\bar \rA') \to C ^0 (S \rA' \cup \bar \rA' |_{p_0}) \to 0 ,$$
the joint symbol $\bj _1 (A + \varPsi'') \in  C ^0 (S \rA' \cup \bar \rA' |_{p_0})$ defines a class 
$\partial [\bj_1 (A + \varPsi'')] \in \bbK _0 (C ^0 (\rA' |_{\rM _0})) $.
Using the anchor map $\nu$ to identify $\rA' |_{\rM_0}$ with $T^* \rM_0$, 
the computation of the Fredholm index of $A + \varPsi''$ is then reduced to the compactly supported Atiyah-Singer index theorem:
\begin{equation}
\label{FredholmAS}
\partial [A + \varPsi''] = (-1)^{\dim \rM} \langle \ch _0 (\partial [\bj_1 (A + \varPsi'')] ) \wp ^* \td (T _\bbC \rM _0 ) , [T \rM _0] \rangle,
\end{equation}
where
$\ch _0 $ denotes the even Chern character, $\td (T _\bbC \rM _0 )$ denotes the Todd class, 
and $[T \rM _0]$ denotes the fundamental class.

We turn to the $\bbK$-theoretic index for general elliptic operators.
We suppose further that $\rA$ is an oriented manifold.
\begin{thm}
\label{TopIndex}
Let $\varPsi$ be a 0-order elliptic operator.
Then
$$ \partial [\varPsi] = (-1)^{\dim \rM} \langle \ch _0 [\sigma (\varPsi)] \wp ^* \td (\rA _\bbC ) , [\rA |_{\rM _0}] \rangle.$$
\end{thm}
\begin{proof}
There exists an operator of the form $A + \varPsi''$ as in \eqref{FredholmAS} such that 
\begin{align*}
\partial [\sigma (\varPsi )] & = \partial [\sigma (A + \varPsi'')] & \in \bbK _0 (C^0 (\rA)), \\
\partial [\varPsi ] & = \partial [A + \varPsi''] & \in \bbK _0 (C^* (\cG)).
\end{align*}
The principal and joint symbol of $A + \varPsi''$ is related by the morphism of short exact sequences
$$
\begin{CD}
0 \to C^0 (T^* \rM _0 ) @>>> C^0 (\bar \rA') @>>> C^0 (S \rA' \cup \bar \rA' |_{p_0}) \to 0 \\
@VV{\iota}V @| @VVV \\
0 \to C^0 (\rA') @>>> C^0 (\bar \rA') @>>> C^0 (S \rA') \to 0
\end{CD}
$$
which implies 
$$ \partial [\sigma (A + \varPsi'')] = \partial [\bj _1 ((A + \varPsi'')|_{S \rA'})] = \iota _* \partial [\bj _1 (A + \varPsi'')]. $$
By assumption both the Todd class and orientation of $T \rM _0$ can be made to be induced by $\rA$.
It follows that
$$ \langle \ch _0 (\partial [\bj_1 (A + \varPsi'')] ) \wp ^* \td (T _\bbC \rM _0 ) , [T \rM _0] \rangle 
= \langle \iota \ch _0 (\partial [\bj_1 (A + \varPsi'')] ) \wp ^* \td (\rA _\bbC ) , [\rA |_{\rM _0}] \rangle .$$
Hence the theorem follows from the naturality of the Chern character and Equation \eqref{FredholmAS}.
\end{proof}

\subsection{Index formula using renormalized trace}
\label{OddRenorm}
If $\dim \rM _1 \geq 1$, the arguments in the previous section does not work.
Instead, we adopt the renormalization arguments used in \cite{Nistor;Hom2}, under additional assumptions.

Let $r := d (\cdot , \rM _1) $ be the Riemannian distance function on $\rM$ with respect to some Riemannian metric.
We assume that there is are coordinate charts $(U _\alpha , \bx _\alpha )$ such that 
$ \cup _\alpha U _\alpha \supset \rM _1 $ and on each $U _\alpha $ the image of the anchor map is spanned by
\begin{equation}
\label{StructuralVF}
r ^N\frac{\partial }{\partial x _1} , \cdots, r^N \frac{\partial }{\partial x _q},
\frac {\partial }{\partial y_1} , \cdots, \frac {\partial }{\partial y_{\dim \rM _1}} .
\end{equation}
Note that these vector fields are smooth at $0$. That implies $N $ is necessarily even.
In cylindrical coordinates the structural vector fields is spanned by
$$ r ^N \frac{\partial }{\partial r} , r ^{N -1} \frac{\partial}{\partial \theta _1}, \cdots, r ^{N -1} \frac{\partial}{\partial \theta _{q-1}},
\frac {\partial }{\partial y_1} , \cdots, \frac {\partial }{\partial y_{\dim \rM _1}}. $$

The Riemannian density of the induced metric on $\rM _0$ is of the form 
$$ r ^{- q N} \mu _\rM $$
for some smooth density $\mu _\rM$ on $\rM$ restricted to $\rM _0$.
Without loss of generality we shall assume $\mu _\rM$ is the Lebesgue measure with respect to each $\bx _\alpha$.

By the same arguments as \cite[Lemma 7]{Nistor;GeomOp} we have:
\begin{lem}
For any boundary groupoid of the form $\cG = (\rM _0 \times \rM _0) \cup (\bbR ^q \times \rM _1 \times \rM _1)$,
$$ C^* (\cG) \cong C^* _r (\cG).$$
\end{lem}
Therefore by Lemma \ref{FunctCal},
the space of Schwartz kernels $\varPsi ^0 (\cG , \rE) \oplus S ^* (\cG)$ is also closed under holomorphic functional calculus in $\fU^* (\cG)$.
In particular, $\cup _m \varPsi ^m (\cG, \rE) \oplus S ^* (\cG)$ satisfies all axioms in \cite[(i)-(vii),($\sigma$),($\psi$)]{Nistor;CplxPwr}.
It follows from \cite[Theorem 7.2]{Nistor;CplxPwr} that
for any uniformly supported elliptic, positive definite operator $Q$ of order 1 
(for example $Q = (\varDelta + I) ^{\frac{1}{2}}$, see \cite{Nistor;GeomOp}),
the complex powers $Q ^z$ are well defined as pseudo-differential operators lying in 
$\varPsi ^{-z} (\cG , \rE) \oplus \cS (\cG)$. 

Given any $\varPsi \in \Psi ^m (\cG)$, 
$\varPsi Q ^{- \tau} $ is given by a sufficiently smooth kernel whenever the real part of $\tau$ is sufficiently large.
Hence it makes sense to consider
$$ Z(\varPsi , \tau, z) := \int _{\rM _0} \tr (r ^{z - q N} \varPsi Q ^{- \tau} ) \mu _\rM , \quad \tau, z \in \bbC$$
(where we regard $r ^{z - q N}$ as the operator defined by multiplying by a scalar function).
By \cite[Lemma 1]{Nistor;Hom2}, $Z (\varPsi , \tau, z)$ is a meromorphic function of $\tau $ and $z$, 
with at most simple pole at $\tau = z = 0 $.
Hence one defines:
\begin{dfn}
Given a holomorphic family $A (\tau , z)$ of pseudo-differential operators, its renormalized trace is defined to be
$$ \widehat \tr (A) := \frac{\partial ^2}{\partial \tau \partial z} \big|_{\tau = 0, z = 0} 
\int _{\rM _0} \tau z \tr \big(r ^{z - q N} A (\tau , z) Q ^{- \tau} \big |_{\rM } \big) \mu _\rM  .$$
\end{dfn}

In view of Remark \ref{MainRem}, to derive an index formula it suffices to consider the Fredholm index.
Suppose now that $D \in \rD ^1 (\cG , \rE)$ is a fully elliptic, order 1 differential operator.
There exists a pseudo-differential parametrix $\varPhi$ of $D$ 
up to finite rank operators and vanishing up to infinite order at $\cG _1$ \cite{So;FullCal}.
By the Cauldron formula, the Fredholm index of $D$ equals
\begin{equation}
\label{Cauldron}
\int _{\rM _0} r ^{-q N} \tr ([D, \varPhi]) \mu _\rM = \widehat \tr ([D, \varPhi]) .
\end{equation}
Following the arguments of \cite[Proposition 8]{Nistor;Hom2}, 
we have by direct calculation
\begin{equation}
\label{ResEq}
r ^z [D, \varPhi] Q ^{- \tau} = (r ^z D r ^{- z} - D) r ^z \varPhi Q ^{- \tau}
+ r ^z \varPhi (Q ^{- \tau} D Q ^\tau - D) Q ^{- \tau} + [D, r ^z \varPhi Q ^{- \tau}].
\end{equation}
The last term of \eqref{ResEq} has vanishing trace. 
Consider the first term of \eqref{ResEq} (this corresponds to the $\eta$ term in \cite[Equation (18)]{Nistor;Hom2}.
However, unlike the manifold with boundary case, we have:
\begin{lem}
\label{OddEta}
The renormalized trace
\begin{equation}
\frac{\partial ^2}{\partial \tau \partial z} \big|_{\tau = 0, z = 0} 
\tau z \int _{\rM _0} \tr ( r ^{- q N} (r ^z D r ^{- z} - D) r ^z \varPhi Q ^{- \tau}) \mu _\rM  = 0 .
\end{equation}
\end{lem}
\begin{proof}
The proof reduces to a local one by considering partition of unity.
We parameterize each $U _\alpha$ using cylindrical coordinates and let $\sigma$ be the cylindrical anti-pole map, i.e.,
$$ (r , \sigma (\theta ), y ) := (- (r, \theta), y) .$$

By assumption $D$ is the composition of one of the vector fields in \eqref{StructuralVF} and tensors.
Taking local trivialization, one writes $D$ in cylindrical coordinates:
\begin{align*}
D = & D _0 (r, \theta, y) r ^N \frac{\partial }{\partial r} +
D _1 (r, \theta, y) r ^{N -1} \frac{\partial}{\partial \theta _1} + \cdots
+ D _{q-1} (r, \theta, y) r ^{N -1} \frac{\partial}{\partial \theta _{q-1}} \\
& + D' _1 (r, \theta, y) \frac{\partial }{\partial y_1} + \cdots + D' _{\dim \rM _1} (r, \theta, y) \frac{\partial }{\partial y_{\dim \rM _1}} + D''(r, \theta, y).
\end{align*}
Consider $(r ^z D r ^{- z} - D) r ^z$. One has
$$ \big[ r ^N \frac{\partial}{\partial r} , r ^z \big] = z r ^{z + N - 1},$$
which implies 
$$ \big(r ^z (r ^N \frac{\partial}{\partial r}) r ^{- z} - r ^N \frac{\partial}{\partial r} \big) r ^z \varPhi Q ^{- \tau} 
= z r ^{z + N -1} \varPhi Q ^{- \tau};$$
while other commutators vanish.
It follows that
$$ (r ^z D r ^{- z} - D) r ^z \varPhi Q ^{- \tau} = z r ^{z-1} r ^N D _0 \varPhi Q ^{- \tau} .$$
Expand $D _0 $ as a Taylor series 
$$ D _0 = A _0 (\theta, y) + r A _1 (\theta, y) + \cdots ,$$
for some matrices $A _j (\theta, y)$.
Since $D _0 r ^N \frac{\partial }{\partial r}$ is a smooth differential operator at $\{ r = 0 \}$, it follows that
$$ A _j (\sigma (\theta), y) = - (-1) ^j A _j (\theta , y), \quad j=0, 1, 2, \cdots .$$
Similarly, for any $\tau$ with sufficiently large real part, 
$\varPhi Q ^{- \tau}$ is given by a sufficiently smooth kernel.
It follows that $\tr (D _0 \varPhi Q ^{- \tau})$ has Taylor expansion with respect to $r$:
$$\tr (D _0 \varPhi Q ^{- \tau}) = F _0 (\theta , y) + F _1 (\theta , y) r + \cdots .$$
satisfying
\begin{equation}
\label{Parity1}
F _j (\sigma (\theta), y) = -(-1) ^j F _j (\theta , y), \quad j = 0, 1, 2, \cdots .
\end{equation}
In cylindrical coordinates, denoting the volume measure of the $q-1$ dimensional sphere by $d \theta$,
the integral in \eqref{OddEta} becomes
$$ \int z r ^{z-1 + N} r ^{- q N + q - 1} \tr (D _0 \varPhi Q ^{- \tau}) d r d \theta d y .$$
One replaces the trace factor by its Taylor expansion, integrate with respect to $r$ and take the constant term to get:
\begin{equation}
\label{PwrInt}
\int F _{(N-1)(q-1)} (\theta , y) d \theta d y.
\end{equation} 
Since $(q - 1)(N - 1) $ is even, it follows from parity \eqref{Parity1} that the $\theta$ integral in \eqref{PwrInt} also vanishes 
for all $\tau$ with sufficiently positive real part. 
Hence the lemma follows by meromorphic extension.
\end{proof}

It remains to consider the second term of \eqref{ResEq}.
This term has been computed in \cite[Proposition 12]{Nistor;Hom2}, and we shall briefly recall the result here.
\begin{lem}
One has
$$ \widehat \tr (\varPhi (Q ^{- \tau} D Q ^\tau - D) Q ^{- \tau}) = \int _{\rM _0} r ^{z - q N} a _0 \mu _\rM \Big| _{z = 0},$$
where $a _0 $ is the constant term in the $t \to 0 $ asymptotic expansion of the trace of heat kernel 
$$\tr \big( (e ^{-t D' (D') ^*} + e ^{-t (D') ^* D'}) \big|_\rM \big).$$
\end{lem}
\begin{proof}
We fix 
\begin{align*}
Q := Q _1 := & (D D ^* + \varPi _{\ker (D^*)}) ^{\frac{1}{2}} \\
Q _2 := & (D ^* D + \varPi _{\ker (D)}) ^{\frac{1}{2}} \\
\varPhi := & D ^* Q _1 ^{-2} = Q _2 ^{-2} D ^*,
\end{align*}
where $\varPi _{\ker (D)}, \varPi _{\ker (D^*)} $ respectively denote orthogonal projection (in $L ^2 (\rM _0)$) onto $\ker (D)$, and $\ker (D^*)$.
Note that for any $\tau \in \bbC$,
$$ Q _1 ^ {-\tau } D = D Q _2 ^{- \tau }.$$
It follows that for any $\tau, z \in \bbC$:
\begin{align*}
\varPhi (Q ^{- \tau} D Q ^\tau - D) Q ^{- \tau}
= & \: \varPhi (Q _1 ^{- \tau } D - D Q _1 ^{- \tau}) \\
= & \: \varPhi D (Q _2 ^{- \tau } - Q _1 ^{- \tau}) \\
\widehat \tr (\varPhi (Q ^{- \tau} D Q ^\tau - D) Q ^{- \tau}) 
= & \widehat \tr (Q _2 ^{- \tau } - Q _1 ^{- \tau}).
\end{align*}
Since $\tau ^{-1} (Q _2 ^{- \tau } - Q _1 ^{- \tau})$ is a holomorphic family, 
by \cite[Lemma 1]{Nistor;Hom2}, 
$\int _{\rM _0} \tr (r ^{z - q N} (Q _2 ^{- \tau } - Q _1 ^{- \tau})) \mu _\rM$ is holomorphic at $\tau = 0$.
Using the Mellin transform one gets for any $\tau $ with sufficiently large real part
$$ \tr (Q _2 ^{- \tau } - Q _1 ^{- \tau} |_\rM)
= \frac{1}{\Gamma (\tau / 2)} \int _0 ^\infty t ^{\frac{\tau}{2} - 1} 
\tr \big( e ^{- t (D D ^* + \varPi _{\ker (D^*)}) } - e ^{- t (D ^* D + \varPi _{\ker (D^*)}) } |_\rM \big) d t .$$
Since 
\begin{align*}
e ^{- t (D D ^* + \varPi _{\ker (D^*)}) } = & e ^{- t D D ^* } + (e ^t - 1) \varPi _{\ker (D^*) } \\
e ^{- t (D ^* D + \varPi _{\ker (D^*)}) } = & e ^{- t D ^* D} + (e ^t - 1) \varPi _{\ker (D^*) },
\end{align*}
the claim follows.
\end{proof}

Summarizing the calculations above, we obtain:
\begin{thm}
\label{RenormIndex}
For any elliptic, differential operator $D \in \rD ^1 (\cG , \rE)$,
the index of $D$ in $\bbK _0 (C^* (\cG))$ equals
\begin{equation}
\partial ([D (D D^* + \id )^{- \frac{1}{2}} ])
= \int _{\rM _0} r ^{z - q N} a _0 \mu _\rM \Big |_{z = 0},
\end{equation}
where $a _0 $ is the constant term in the $t \to 0 $ asymptotic expansion of the trace of heat kernel.
\end{thm}

Note that when $\rM _1$ is a point, Theorem \ref{RenormIndex} is the same as Theorem \ref{TopIndex}.

\section{The $q = 1$ case}
We turn to consider boundary groupoids with dimension 1 isotropy subgroups.

\subsection{Connected manifold with connected boundary}
Let $\rM = \rM _0 \cup \rM _1$ be a connected manifold with boundary,
where we denote by $\rM _0$ and $\rM _1$ respectively the interior and boundary of $\rM$.
We assume $\rM _1$ is also connected.
Then 
$$\cG = (\rM _0 \times \rM _0 ) \cup (\bbR \times \rM _1 \times \rM _1).$$
By the proof of Lemma \ref{Morita}, $\cG$ is Morita equivalent to some groupoid of the form
$$ \tilde \cG := ((0, 1) \times (0, 1)) \cup \bbR ,$$
and by Lemma \ref{OddClassSp} 
$$\bbK _\bullet (C^* (\tilde \cG)) \cong \bbK _\bullet (C^0 [0, 1)) \cong \{ 0 \} .$$
We conclude that:
\begin{lem}
One has
$$ \bbK _0 (C^* (\cG)) \cong \{ 0 \} , \quad \bbK _1 (C ^* (\cG)) \cong \{ 0 \} ,$$
and the $\bbK$-theoretic index is trivial.
\end{lem}

\begin{rem}
From \eqref{Comp3} we have the exact sequence \cite{Nistor;BoundaryK}
$$ 0 \to \bbK _1 (C^* (\cG)) \to \bbZ \xrightarrow{\partial} \bbZ \to \bbK _0 (C^* (\cG)) .$$
It follows that the connecting map $\partial$ is an isomorphism.
More concretely we have
\begin{equation}
\label{Incidence}
\partial (1) = 1 .
\end{equation}
\end{rem}

Suppose that $\rM _1 = r ^{-1} (0) $ for some defining function $r \in C^\infty (\rM)$,
and furthermore that the structural vector fields near $\rM _1$ are spanned by
$$ r ^N \frac{\partial }{\partial r}, \frac{\partial}{\partial y_1}, \cdots , \frac{\partial}{\partial y _{\dim \rM _1}} $$
for some integer $N$,
then formulas computing the Fredholm index of an elliptic operator $D$ that is invertible on $\cG _1$ is well known ---
one of which is by adapting Section \ref{OddRenorm}. We just briefly recall the result here. 
\begin{lem}
\label{APS}
\cite[Proposition 12, Proposition 13]{Nistor;Hom2}
For any order 1 elliptic differential operator $D$ on the groupoid 
$\cG = (\rM _0 \times \rM _0) \cup (\bbR \times \rM _1 \times \rM _1)$,
that is invertible on $\bbR \times \rM _1 \times \rM _1$,
the Fredholm index of $D (D D^* + \id )^{- \frac{1}{2}}$ equals
\begin{equation}
\partial ([D (D D^* + \id )^{- \frac{1}{2}} ])
= \int _{\rM _0} a _0 \mu _\rM + \eta (D),
\end{equation}
where $a _0 $ is the constant term in $t \to 0 $ asymptotic expansion of the trace of heat kernel 
$$\tr \big( (e ^{-t D D ^*} + e ^{-t D^* D}) \big|_\rM \big),$$
and 
$$\eta (D) := \frac{\partial ^2}{\partial \tau \partial z} \big|_{\tau = 0, z = 0} 
\int _{\rM _0} \tr \big(r ^{z - N} ( D - r ^{-z} D r ^z) ( D D^* + \id )^{-\frac{1}{2} - \tau} \big) \mu _\rM $$ 
depends only on $D |_{\bbR \times \rM _1 \times \rM _1}$.
\end{lem}

\subsection{Manifold partitioned by an interior hyper-surface}
We turn to our second example.
Let $\rM $ be a compact manifold without boundary, $\rM _1$ is a closed sub-manifold with co-dimension 1.
We suppose $\rM = \rM' _0 \cup \rM _1 \cup \rM'' _0 $,
where $\rM' _0 , \rM'' _0 $ are open and connected.
We furthermore suppose that the closures, $\bar \rM' _0 , \bar \rM'' _0 $ are manifolds with boundary $\rM _1$, 
as in the last sub-section.
In particular, our assumption implies the normal bundle of $\rM _1$ is trivial.

We consider the subspace of vector fields tangential to $\rM _1$.
The corresponding Lie groupoid is
$$ \cG = (\rM' _0 \times \rM' _0) \cup (\rM'' _0 \times \rM'' _0) \cup (\bbR \times \rM _1 \times \rM _1),$$
it contains
$$ \cG' := (\rM' _0 \times \rM' _0) \cup (\bbR \times \rM _1 \times \rM _1),
\quad \cG'' = (\rM'' _0 \times \rM'' _0) \cup (\bbR \times \rM _1 \times \rM _1)$$
as invariant subspaces.

We first compute $\bbK _0 ( C ^* (\cG))$.
For the composition series of $\cG$ we have
$$ C^* ((\rM' _0 \times \rM' _0) \cup (\rM'' _0 \times \rM'' _0)) 
= C^* (\rM' _0 \times \rM' _0) \oplus C^* (\rM'' _0 \times \rM'' _0) 
= \cK \oplus \cK ,$$
where each copy of $\cK$ corresponds to one component.
Therefore \eqref{Comp4} becomes the exact sequence:
\begin{equation}
\label{CD2.2}
\bbK _1 (C^* (\bbR \times \rM _1 \times \rM _1 )) \cong \bbZ \xrightarrow{\partial {}} \bbZ \oplus \bbZ 
\xrightarrow{\iota _*} \bbK _0 (C^* (\cG)) \to 0.
\end{equation}
We compute the connecting map $\partial$ using the morphism of exact sequences
$$
\begin{CD}
0 @>>> 
\begin{array}{c}
C^* (\rM' _0 \times \rM' _0) \\
\oplus C^* (\rM'' _0 \times \rM'' _0) 
\end{array}
@>>> C^* (\cG) @>>> C^* (\cG ) / (\cK \oplus \cK) @>>> 0 \\
@. @| @VVV @VVV @. \\
0 @>>> 
\cK \oplus \cK @>>> 
\begin{array}{c}
C^* (\cG' ) \\
\oplus C^* (\cG'')
\end{array} 
@>>> 
\begin{array}{c}
C^* (\cG' ) / \cK \\
\oplus C^* (\cG'' ) / \cK
\end{array}
@>>> 0
\end{CD}
$$
where the middle column map is defined by restriction to $\cG' , \cG'' \subseteq \cG$.
From which one gets the commutative diagram:
$$ 
\begin{CD}
\bbK _1 (C^* (\bbR \times \rM _1 \times \rM _1 ))
@>{\partial {}}>> \bbK _0 (C^* (\rM' _0 \times \rM' _0)) \oplus \bbK _0 (C^* (\rM'' _0 \times \rM'' _0)) \\
@VVV @| \\
\bbK _1 (C^* (\cG' ) / \cK) \oplus \bbK _ 1 (C^* (\cG'' ) / \cK) @>{\partial \oplus \partial {}}>> \bbZ \oplus \bbZ
\end{CD}
$$
Here, the bottom row is just two copies of the connection map in \eqref{Incidence}.
On the other hand, we have
$$ C^* (\bbR \times \rM _1 \times \rM _1 ) = C^* (\cG' ) / C^* (\rM' _0 \times \rM' _0) = C^* (\cG'' ) / C^* (\rM'' _0 \times \rM'' _0), $$
and the left column is just two copies of the identity map.
It follows that
$$\partial (1) = 1 \oplus 1 .$$
Hence $\bbK _0 ( C^* (\cG ) ) \cong \bbZ $ and $\iota _* :  \bbZ \oplus \bbZ \to \bbK _0 (C ^* (\cG)) $ in \eqref{CD2.2} is realized as
\begin{equation}
\label{AntiInd}
\iota _* (1 \oplus 0 ) = 0, \iota _* (0 \oplus 1) = -1 .
\end{equation}

We turn to study the index.
First, given any elliptic pseudo-differential operator $\varPsi$ that is invertible on $\cG _1$,
we consider the index of $\varPsi$ in $ \bbK _0 ( C^* ((\rM' _0 \times \rM' _0) \cup (\rM'' _0 \times \rM'' _0))) $.
One has the homomorphism of short exact sequences:
$$
\begin{CD}
0 \to 
\begin{array}{c}
C^* ((\rM' _0 \times \rM' _0) \\
\quad \cup (\rM'' _0 \times \rM'' _0)) 
\end{array}
@>>> \fU (\cG) 
@>>> \fU (\cG) \Big/ 
\begin{array}{c}
C^* ((\rM' _0 \times \rM' _0) \\
\quad \cup (\rM'' _0 \times \rM'' _0)) 
\end{array} 
\to 0 \\
@| @VVV @VVV \\
0 \to 
\begin{array}{c}
C^* (\rM' _0 \times \rM' _0)\\
\oplus C ^* (\rM'' _0 \times \rM'' _0)
\end{array}
@>>> 
\begin{array}{c}
\fU (\cG') \\
\oplus \fU (\cG'')
\end{array}
@>>> 
\begin{array}{c}
\fU (\cG') / C^* (\rM' _0 \times \rM' _0) \\
\oplus \fU (\cG'') / C^* (\rM'' _0 \times \rM'' _0)
\end{array}
\to 0
\end{CD}
$$
where the middle map is given by (the extension of) restricting the pseudo-differential operator on $\cG$ 
to the invariant subspaces $\cG'$ and $\cG''$,
and it is clear the restriction map induces a map for the right column.

By the naturality of the index map, one has
\begin{equation}
\label{CD3}
\begin{CD}
\bbK _1 \Big(\fU (\cG) \Big/ 
\begin{array}{c}
C^* ((\rM' _0 \times \rM' _0) \\
\quad \cup (\rM'' _0 \times \rM'' _0)) 
\end{array} 
\Big) 
@>{\partial}>> 
\bbK _0 \Big( 
\begin{array}{c}
C^* ((\rM' _0 \times \rM' _0) \\
\quad \cup (\rM'' _0 \times \rM'' _0)) 
\end{array} 
\Big)  \\
@VVV @| \\
\begin{array}{c}
\bbK _1 (\fU (\cG') / C ^*(\rM' _0 \times \rM' _0)) \\
\oplus \bbK _1 (\fU (\cG'') / C^* (\rM'' _0 \times \rM'' _0)) 
\end{array}
@>{\partial \oplus \partial}>> \bbK (\cK \oplus \cK) \cong \bbZ \oplus \bbZ
\end{CD}
\end{equation}
where ${\partial} \oplus {\partial}$ is just the Fredholm index maps for each copy.

To conclude, the index of $\varPsi$ in $\bbK _0 ( C^* ((\rM' _0 \times \rM' _0) \cup (\rM'' _0 \times \rM'' _0))) $
is the direct sum of the Fredholm indexes of $\varPsi |_{\cG '}$ and $\varPsi |_{\cG ''}$
(and the Fredholm index of $\varPsi $ is clearly the arithmetic sum of the two indexes in $\bbZ$).

We turn to derive a formula for the index in $\bbK _0 (C^* (\cG))$
of any order 1 elliptic pseudo-differential operator $D$ that is invertible on $\bbR \times \rM _1 \times \rM _1$. 
We use Equation \eqref{Surj} which in this case reads:
\begin{equation}
\begin{CD}
\bbK _1 (\fU / C^* ((\rM' _0 \times \rM' _0) \cup (\rM'' _0 \times \rM'' _0))) @>>> \bbK _0 (\cK \oplus \cK) = \bbZ \oplus \bbZ \\
@VVV @VV{\iota _*}V \\
\bbK _1 (\fU / C ^* (\cG)) @>{\partial {}}>> \bbK _0 ( C ^* (\cG)) 
\end{CD}
\end{equation}
where the right column is just \eqref{AntiInd}. It follows from Lemma \ref{APS} that 
\begin{equation}
\partial ( [D (D (D)^* + \id )^{- \frac{1}{2}} ])
= \int _{\rM' _0} a' _0 - \int _{\rM'' _0} a'' _0 \in \bbK _0 (\cC ^* (\cG)) \cong \bbZ,
\end{equation}
where $a' _0, a'' _0$ are respectively the constant term in $t \to 0 $ asymptotic expansion of the traces 
of the heat kernels 
$$ e ^{- t D D ^* }|_{\rM' _0 \times \rM' _0 } +  e ^{- t D ^* D }|_{\rM' _0 \times \rM' _0 }
\text{ and }
e ^{- t D D ^*} |_{\rM'' _0 \times \rM'' _0 } + e ^{- t D ^* D} |_{\rM'' _0 \times \rM'' _0 } .$$

We turn to general case when $D \in \rD ^1 (\cG, \rE) + S ^* (\cG)$ and is elliptic.
From Section 4, there exists a Fredholm operator $\tilde D $ of the from
$$ \tilde D = D + R , R \in S ^* (\cG)$$ 
such that 
$$ \partial ([\tilde D (\tilde D \tilde D^* + \id )^{- \frac{1}{2}} ])
= \partial ([D (D D^* + \id )^{- \frac{1}{2}} ]) \in \bbK _0 (C^* (\cG)).$$
Since $\tilde D$ and $D$ have the same principal symbol the 
$t \to 0$ asymptotic expansions of the traces their heat kernels have the same constant.
It follows that:
\begin{thm}
\label{OneInd}
One has the index formula for any order 1 elliptic pseudo-differential operator $D \in \rD ^1 (\cG, \rE) + S ^* (\cG)$:
$$ \partial ([D (D (D)^* + \id )^{- \frac{1}{2}} ])
= \int _{\rM' _0} a' _0 - \int _{\rM'' _0} a'' _0 \in \bbK _0 (C^* (\cG)) \cong \bbZ.$$
\end{thm}

\section{Isotropy subgroup of even dimension}
In this section we consider $\cG$ of the form
$$\cG = (\rM _0 \times \rM _0 ) \cup (\bbR ^q \times \rM _1 \times \rM _1), \text{ $q$ even.}$$
From Equation \eqref{Comp3}, one readily computes:
\begin{thm}
\label{EvenK}
One has
$$ \bbK _0 (C^* (\cG)) \cong \bbZ \oplus \bbZ , \quad \bbK _1 (C^* (\cG)) \cong \{ 0 \}.$$
\end{thm}
The above isomorphism can be made more explicit.
Fix any $[ \varPsi ] \in \bbK _0 (C^* (\cG _0)) $ such that $\pi _* ([ \varPsi ]) = 1 \in \bbK _0 (C^* (\bar \cG _1))$.
Then $\bbK _0 (C^* (\cG))$ is the free Abelian group generated by $\iota _* (1) $ and $[ \varPsi ]$.

\subsection{Index formulas}

Recall that the map $\pi _* : \bbK _0 (C^* (\cG)) \to \bbK _0 (C^* (\cG _1))$ is induced by restricting the kernel to $\bar \cG _1$.
Index formulas for $\bar \cG _1 = \bbR^q \times \rM _1 \times \rM _1$ are well known.
For example, by \cite[Theorem 8]{Nistor;BoundaryK}, we have
\begin{equation}
\partial (D (D D^* + 1) ^{-\frac{1}{2}}) = (-1) ^{\dim \rM _1} \langle \ch _0 [\sigma (D)|_{\rM _1}] \wp ^*\td (\rM _1), [S^* \rM _1] \rangle,
\end{equation}
where $\td (\rM _1)$ denotes the Todd class, $\wp : S^* \rM _1 \to \rM _1 $ is the fiber bundle projection and $[S^* \rM _1]$ is the fundamental class.

\begin{rem}
Alternatively, the Lie algebroid of $\cG_1$ is $\bbR^q \oplus T \rM_1$.
Therefore $\wedge ^{\dim \rM} T^* \rM \otimes \wedge ^{\dim \rM} \rA$ is a trivial vector bundle.
Moreover the product of any density on $\rM _1$ and the invariant volume form on $\bbR ^q$ defines an invariant measure $\varOmega$.
It follows one can apply \cite[Theorem B]{Pflaum;GpoidLocalizedInd} to obtain the formula for any elliptic differential operator:
\begin{align}
\label{GpoidLocalizedInd}
\langle \chi (1|_{\rM _1}) &, \partial (D (D D^* + 1) ^{-\frac{1}{2}}) \rangle \\ \nonumber
=& \frac{1}{(2 \pi \sqrt{-1}) ^{\dim \rA}} \int _{\rA' | _{\rM _1} } \hat A (\pi ^! \rA |_{\cG _1} ) \ch _0 [\sigma (D |_{\cG _1})] \varOmega _{\pi ^! \rA} .
\end{align}
Here $\langle \:, \, \rangle$ denotes the pairing between cyclic homology and $\bbK$-theory, and $\chi$ denotes the van-Erp map.

However, in general an invariant measure may not exist, and \cite{Pflaum;GpoidLocalizedInd} cannot be applied.
The obstruction to the existence of such measure is studied in \cite{Lu;Modular,Mestre;Thesis}.
\end{rem}

As for the Fredholm index part, if $\rM _1$ is a point and $\rA $ is orientable, 
then Theorem \ref{TopIndex} also applies to compute the Fredholm part, i.e., one has
\begin{thm}
Let $\cG =  (\rM _0 \times \rM _0 ) \cup \bbR ^q $, $D $ be a 1 order elliptic operator, then the index $\partial ([D (D D^* + 1) ^{-\frac{1}{2}}])$ is given by:
\begin{align*}
\partial ([D (D D^* + 1) ^{-\frac{1}{2}}]) 
= & (-1)^{\dim \rM} \langle \ch _0 [\sigma (\varPsi)] \wp ^* \td (\rA _\bbC ) , [\rA |_{\rM _0}] \rangle \\
&\oplus (-1) ^{\dim \rM _1} \langle \ch _0 [\sigma (D) |_{\rM _1}\ \wp ^*\td (\rM _1), [S^* \rM _1] \rangle.
\end{align*}
\end{thm}

One can also consider renormalized index similar to Section \ref{OddRenorm}.
However, the arguments in Lemma \ref{OddEta} cannot be used to show that the eta term vanished.
On the other hand because the natural incursion $\bbK _0 (\cK) \to \bbK _0 (C^* (\cG))$ is injective (by Equation \eqref{Comp3}),
we know a-prior that the eta term only depends on the principal symbol!

\section{Some examples with $r > 1$}
In this section we compute some more complicated examples.
Let $\cG$ be a strongly central boundary groupoid as in Definition \ref{Central}.
We further assume that $\rG ^0 _1, \cdots, \rG ^0 _r$ are all exponential.

\subsection{Even dimensional isotropy subgroups}
In this section, we assume that
$\dim (\rG _i), i = 1, \cdots, r$ are all even.
We have exact sequences form the composition series:
\begin{align}
\label{EvenIterated}
\nonumber
0 \to C^* (\rM _0 \times \rM _0 )\cong \cK & \to C ^* (\cG ) \to C ^* (\rG ^0 _1 \times \cG ^0 _1) \to 0 \\
0 \to C^* (\rM _1 \times \rM _1 )\cong \cK & \to C ^* (\cG ^0 _1 ) \to C ^* (\rG ^0 _2 \times \cG ^0 _2) \to 0 \\ \nonumber
& \cdots \\ \nonumber
0 \to C^* (\rM _{r-1} \times \rM _{r-1} )\cong \cK & \to C ^* (\cG ^0 _{r -1} ) \to C ^* (\rG ^0 _r \times \cG ^0 _r) \to 0.
\end{align}
Note that all $\rG ^0 _i$ are even dimensional.
The six term exact sequence induced by the last exact sequence in \eqref{EvenIterated} is just \eqref{Comp3}.
Therefore
$$ \bbK _0 (C ^* (\cG ^0 _{r-1} )) \cong \bbZ \oplus \bbZ , \quad \bbK _1 (C ^* (\cG ^0 _{r -1} )) \cong \{ 0 \} .$$
Hence one calculates inductively
$$
\begin{CD}
\bbK _0 (\cK) \cong \bbZ @>>> \bbK _0 (C^* (\cG ^0 _{r-k-1})) @>>> \bbK _0(C^* (\cG ^0 _{r-k})) \cong \bbZ ^{k+1} \\
@AAA @. @VVV \\
\bbK _1 (C^* (\cG ^0 _{r-k})) \cong \{ 0 \} @<<< \bbK _1 (C^* (\cG ^0 _{r-k-1})) @<<< \bbK _1 (\cK) \cong \{ 0 \}
\end{CD}
$$
where $\bbZ ^{k+1}$ denotes $\bbZ \oplus \cdots \oplus \bbZ $ with $k+1$ copies.
It follows that:
\begin{prop}
One has
$$ \bbK _0 (C^* (\cG)) \cong \bbZ ^{r+1} , \quad \bbK _1 (C^* (\cG)) \cong \{ 0 \}.$$
\end{prop}

\begin{exam}
The results in this section clearly applies to the symplectic groupoid in Example \ref{BruhatExam}.
We get
$$ \bbK _0 (C^* (\cG)) \cong \bbZ ^{n + 1} , \quad \bbK _1 (C^* (\cG)) \cong \{ 0 \} .$$
\end{exam}

\subsection{More examples with $r = 2$}
We turn to another example of strongly central groupoid of the form
$$\cG = (\rM _0 \times \rM _0) \cup (\rG _1 \times \rM _1 \times \rM _1) \cup (\rG _2 \times \rM _2 \times \rM _2).$$
By assumption one has $\bar \cG _1 = \rG ^0 _1 \times \cG ^0 _1 $,
where $\cG ^0 _1 = (\rM _1 \times \rM _1) \cup (\rG ^0 _2 \times \rM _2 \times \rM _2)$.
We shall further assume that $\dim \rG ^0 _1$ and $\dim \rG ^0 _2$ are both odd and greater than $1$.

Form the composition series Lemma \ref{Comp}, we have the morphism of exact sequences
\begin{equation}
\begin{CD}
\label{CD8.1}
0 \to C^* (\cG _{\rM _0}) @>>> C^* (\cG) @>>> C^* (\cG _{\rM _1 \cup \rM _2}) \to 0 \\
@VVV @| @VVV \\
0 \to C^* (\cG _{\rM _0 \cup \rM _1}) @>>> C^* (\cG) @>>> C^* (\cG _{\rM _2}) \to 0
\end{CD}
\end{equation}
where the left and right column maps respectively comes from the composition series
\begin{align}
0 & \to C^* (\cG _{\rM _0}) \to C^* (\cG _{\rM _0 \cup \rM _1}) \to C^* (\cG _{\rM _1}) \to 0 \\ \nonumber
0 & \to C^* (\cG _{\rM _1}) \to C^* (\cG _{\rM _1 \cup \rM _2}) \to C^* (\cG _{\rM _2}) \to 0 .
\end{align}
Equation \eqref{CD8.1} induces six term exact sequences
$$
\begin{CD}
\cdots \to \bbK _0 (C^* (\cG _{\rM _1 \cup \rM _2})) \cong \bbK _1 (C^* (\cG ^0 _1)) @>>> \bbK _1 (C^*(\cG _{\rM _0})) \cong \{ 0 \} \to \cdots \\
@VVV @VVV \\
\cdots \to \bbK _0 (C ^* (\cG _{\rM _2})) \cong \bbK _1 (C^* ((\cG ^0 _1) _{\rM _2})) @>>> \bbK _1 (C^* (\cG _{\rM _1 \cup \rM _2})) \to \cdots .
\end{CD}
$$
By Equation \eqref{OddIsom}, the left column in the diagram above is an isomorphism. It follows that bottom row is the zero map.
The rest of the six terms exact sequence induced by the bottom row of \eqref{CD8.1} now reads
\begin{align*}
\bbK _1 (C^* (\cG _{\rM _2})) \cong \{ 0 \} & \to \bbK _0 (C^* (\cG _{\rM _0 \cup \rM _1})) \to \bbK _0 (C^* (\cG)) \to \bbK _0 (C^* (\cG _{\rM _2})) \\
& \xrightarrow{0} \bbK _1 (C^* (\cG _{\rM _0 \cup \rM _1})) \to \bbK _1 (C^* (\cG)) \xrightarrow{0} \cdots .
\end{align*}
Using Theorem \ref{OddMain}, $\bbK _0 (C^* (\cG _{\rM _0 \cup \rM _1})) \cong \bbZ \cong \bbK _1 (C^* (\cG _{\rM _0 \cup \rM _1}))$.
Hence we conclude that
\begin{prop}
One has
$$ \bbK _0 (C^* (\cG)) \cong \bbZ \oplus \bbZ , \quad \bbK _1 (C^* (\cG)) \cong \bbZ.$$
\end{prop}

\section{Some concluding remarks}
In this paper we have computed the $\bbK$-theory and index formulas for some boundary groupoid $C^*$-algebras.
It is obvious that we are not considering the most general cases.
For example it would be interesting to consider manifolds with fibered boundary.
Also it is not clear how a renormalized integral can be defined for boundary groupoids with $r \geq 2$.
Last but not the least the $\eta$ term of the even dimensional case is still mysterious.


\end{document}